\documentclass[12pt,twoside,reqno]{amsart}

\usepackage[titletoc,toc,page]{appendix}
\usepackage{mathtools}

\usepackage{amssymb,amsmath,amstext,amsthm,amsfonts,xcolor, amsthm,amscd}

\usepackage{dsfont}

\usepackage{textgreek}

\usepackage[ansinew]{inputenc} 
\usepackage{graphicx}
\usepackage[mathscr]{eucal}

\usepackage{hyperref, enumerate}

\newcommand{\R}{\mathbb{R}}
\newcommand{\C}{\mathbb{C}}
\newcommand{\N}{\mathbb{N}}
\newcommand{\Z}{\mathbb{Z}}

\newcommand{\T}{\mathbb{T}}
\newcommand{\Su}{\mathbb{S}}
\newcommand{\SL}{{\rm SL}}
\newcommand{\GL}{{\rm GL}}

\newcommand{\Mat}{{\rm Mat}}

\newcommand{\abs}[1]{\bigl| #1 \bigr|} 
\newcommand{\norm}[1]{\lVert#1\rVert} 
\newcommand{\normtwo}[1]{
{\left\vert\kern-0.25ex\left\vert\kern-0.25ex\left\vert #1 
    \right\vert\kern-0.25ex\right\vert\kern-0.25ex\right\vert} }

\newcommand{\ep}{\epsilon}

\newcommand{\om}{\omega}

\newcommand{\qpcocycle}[1]{C_r^{\om} ({#1}, \Mat (m, \C))}

\newcommand{\K}{\mathbb{K}}

\newcommand{\Gr}{{\rm Gr}}
\newcommand{\Pp}{\mathbb{P}}

\newcommand{\id}{{\rm id}}
\newcommand{\F}{\mathbb{F}}

\newcommand\restr[2]{{
  \left.\kern-\nulldelimiterspace 
  #1 
  \vphantom{\big|} 
  \right|_{#2} 
  }}

\theoremstyle{plain}
\newtheorem{theorem}{Theorem}
\newtheorem{proposition}{Proposition}
\newtheorem{corollary}{Corollary}

\theoremstyle{remark}
\newtheorem{remark}{Remark}

\theoremstyle{definition}

\title[Topological obstructions]{Topological obstructions to dominated splitting for ergodic translations on the higher dimensional torus}

\date{}

\begin{document}

\author[P. Duarte]{Pedro Duarte}
\address{Departamento de Matem\'atica and CMAFCIO\\
Faculdade de Ci\^encias\\
Universidade de Lisboa\\
Portugal 
}
\email{pmduarte@fc.ul.pt}

\author[S. Klein]{Silvius Klein}
\address{Departamento de Matem\'atica, Pontif\'icia Universidade Cat\'olica do Rio de Janeiro, Brazil (PUC-Rio)
  }
\email{silviusk@mat.puc-rio.br}

\begin{abstract} 
Consider the space of analytic, quasi-periodic cocycles on the higher dimensional torus. 
We provide examples of cocycles with nontrivial Lyapunov spectrum, whose homotopy classes do not contain any cocycles satisfying the dominated splitting pro\-per\-ty.
This shows that the main result in the recent work ``Complex one-frequency cocycles'' by A. \'Avila, S. Jitomirskaya and C. Sadel does not hold in the higher dimensional torus setting. 
\end{abstract}

\maketitle

\section{Introduction and statements}\label{intro}
It is well known that the homotopy type may prevent a
con\-ti\-nuous linear cocycle over a base dynamical system from being uniformly hyper\-bolic.
In fact, for an $\SL_2(\R)$-valued cocycle over  a circle map,
M. Herman remarked that the topological degrees 
of the base map $T\colon \T\to\T$ and   of the matrix valued function $A\colon \T\to \SL_2(\R)$  provide topological obstructions to the uniform hyperbolicity of the cocycle. More precisely, this obstruction happens when $\deg(T)-1$ does not divide $\deg(A_p)$,
where for any $p\in\Pp(\R^2)$, $A_p\colon\T\to\Pp(\R^2)$ denotes the projective space induced map $A_p(x)=A(x)\,p$ (see~\cite{Viana-book} or~\cite{Bochi-Viana}). 

In sharp contrast with this,  A. \'Avila, S. Jitomirskaya and C. Sadel~\cite{AJS} recently proved that analytic  cocycles $A\colon \T\to\GL_m(\C)$  over irra\-tional translations on the {\em one} dimensional torus $\T$ are always approxi\-mated by cocycles with dominated splitting (a type of uniform projective hyperbolicity), provided the Oseledets filtration is non\-trivial. In particular, every homotopy class of such cocycles contains analytic cocycles with dominated splitting.

In dynamical systems, the  Bochi-Ma\~n\'e  dichotomy   refers to a generic (low regularity) dichotomy
between zero Lyapunov exponents and uniform hyperbolicity, or dominated splitting in higher dimensions. This dichotomy, proved by J. Bochi~\cite{Bochi}, was first announced by R. Ma\~n\'e in the context of $C^1$-area preserving  diffeomorphisms of a surface.
Later J. Bochi and M. Viana generalized it to $C^1$-volume preserving diffeomorphisms of any compact manifold~\cite{Bochi-Viana}. These works~\cite{Bochi,Bochi-Viana}
 also include versions of the dichotomy for classes of $C^0$-cocycles. 
Because the low regularity is essential here,
it is quite surprising that the same type of dichotomy can hold  in~\cite{AJS}  for a class of {\em analytic} cocyles.

\medskip

The purpose of this note is to show that the main result in the aforementioned paper~\cite{AJS} does not hold for cocycles over ergodic translations on the {\em higher} dimensional torus $\T^d$, $d \ge 2$. We obtain this by deve\-loping a simple homological obstruction to the existence of continuous invariant sections of the skew product map induced by the cocycle at the level of the Grassmannian space of a certain dimension.

A somewhat related topic is that of the regularity of the Lyapunov exponents under small perturbations of the cocycle in certain topo\-lo\-gical spaces of cocycles. In~\cite{AJS} the authors prove continuity of the Lyapunov exponents on the space $\qpcocycle{\T}$ of analytic cocycles\footnote{We regard functions on the torus  $\T = \R/\Z$ as $1$-periodic functions on the real line. Then $\qpcocycle{\T}$ is the space of functions $A \colon \T \to \Mat (m, \C)$ admitting a holomorphic extension to the complex strip $\abs{ \Im z} < r$, continuous up to the boundary and endowed with the uniform norm on the strip.}  over irrational translations on the one dimensional torus. Dominated splitting plays a crucial role in their proof, more precisely, the fact that if the Oseledets filtration of the cocycle $A (x)$ is nontrivial, then for small enough $\ep>0$, the complexified  cocycle $A (x+i y)$ has dominated splitting for a.e. $y$ with $\abs{y} < \ep$ (see~\cite[Lemma 4.1]{AJS}). As a consequence of our main result, the analogue of this statement for ergodic translations on the higher dimensional torus does not hold (see Remark~\ref{related to continuity LE}). However, in~\cite{DK-coposim} we established  by other means the continuity of the Lyapunov exponents for analytic cocycles over such translations.
 \medskip
 
 We now introduce the main concepts more formally.

Let $\K=\R$ or $\K=\C$ refer to either the real or the complex field. Let $\T^d = (\R/\Z)^d$ with $d \ge 2$ be the higher dimensional torus. 

A  continuous function
 $A\colon \T^d\to \GL_m(\K)$ and an ergodic  translation
$T:\T^d\to \T^d$ determine the skew-product map
$F\colon\T^d\times\K^m \to \T^d\times\K^m$,
$$F(x,v)=(T x, A(x) v) .$$ 

We call the new dynamical system $F$ a {\em linear cocycle} over the base transformation $T$. Its iterates are $F^{(n)} (x,v)=(T^n x, A^{(n)}(x) v)$, where
$A^{(n)}(x) := A (T^{n-1} x) \ldots A (Tx) \, A (x)$.

Since $T$ is usually fixed, we identify the linear cocycle $F$ with the matrix-valued function $A$, and its iterates $F^{(n)}$ with $A^{(n)}$. 

The {\em Lyapunov exponents} of a linear cocycle $A$ measure the average exponential rate of
growth of the iterates $A^{(n)} (x)$ along the invariant subspaces given by the Oseledets theorem. 

We say that a linear cocycle $A$ has {\em dominated splitting}  with respect to $\K$ (or that its Oseledets decomposition is dominated)
if there exists a continuous $F$-invariant decomposition $\K^m=E_1(x)\oplus \ldots \oplus E_l(x) $,
where $2\leq l \leq m$ and each $E_i$ is an $F$-invariant  continuous $\K$ sub-bundle of the trivial bundle $\T^d\times\K^m$ such that for some
$\lambda>1$, for  any $1\leq i < j \leq l$  and for any unit vectors $v_i\in E_i(x)$, $v_j\in E_j(x)$,
$$ \frac{ \norm{A^{(n)}(x) v_i} }{ \norm{A^{(n)}(x) v_j} } \geq \lambda^n \quad \text{ for all } \; n\in\N.  $$

In particular, as $l\ge2$, the Oseledets decomposition of $A$ is nontrivial (its components are proper subspaces of $\K^m$) so the Lyapunov exponents of $A$ are not all equal. 

For $\SL_2(\R)$-valued cocycles, the dominated splitting pro\-per\-ty is equivalent to uniform hyperbolicity.

Following the terminology in~\cite{AJS}, given $1\leq k<m$,
 we say that a linear cocycle
$A\colon \T^d \to \GL_m(\K)$ is {\em $k$-dominated}  if it admits a domi\-nated decomposition
$\K^m=E^+\oplus E^-$ with $\dim E^+ = k$ and where the Lyapunov exponents of $A\vert_{E^+}$ are strictly larger than all Lyapunov exponents of $A\vert_{E^-}$.

It is clear that  if the linear cocycle $A$ has the dominated spli\-tting  $\K^m=E_1(x)\oplus \ldots \oplus E_l(x) $, then $A$ is $k$-dominated for every dimension $k = \dim (E_1) + \ldots + \dim (E_i)$ with $1\le i \le l-1$.

\medskip

We are now ready to formulate the main result of this paper.

\begin{theorem}
\label{teor main}
Given integers $d\geq 2$ and $1\leq k <m$
there exist analytic quasi-periodic cocycles
$A\colon \T^d\to \GL_m(\C)$ with an invariant 
measurable decomposition $\C^m=E^+\oplus E^-$
such that
\begin{enumerate}
\item $\dim E^+=k$,
\item all Lyapunov exponents of $A\vert_{E^+}$ are positive,
\item all Lyapunov exponents of $A\vert_{E^-}$ are negative,
\item no continuous cocycle $B\colon \T^d\to \GL_m(\C)$ in the homotopy class of $A$ is $k$-dominated.
\end{enumerate}   
\end{theorem}

\begin{remark}
This theorem shows that the dichotomy in~\cite[Theorem 1.1]{AJS} does not hold for analytic quasi-periodic cocycles over a torus $\T^d$ of dimension $d\geq 2$. In fact any sufficiently  small neighborhood $\mathscr{V}$ of $A$ is contained in the homotopy class of $A$. In this neighborhood $\mathscr{V}$, by our continuity result~\cite[Theorem 6.1]{DK-book}, assuming that the translation vector satisfies a generic Diophantine condition,
the Oseledets decomposition $\C^m=E^+\oplus E^-$ persists with $\dim E^+ = k$. However, in view of Theorem~\ref{teor main}, this decomposition is never $k$-dominated.
\end{remark}

Consider now the projective space  $\Pp(\K^{m})$
where the group $\GL_m(\K)$ acts transitively.
More generally let $\Gr_k(\K^m)$ be the Grassmannian 
space of all $k$-dimensional $\K$-linear subspaces of $\K^m$,
which reduces to the projective space when $k=1$.

The cocycle $F$ determines the skew-product map 
$\hat F\colon \T^d\times \Gr_k(\K^m)\to \T^d\times \Gr_k(\K^m)$ defined by
 $\hat F(x, V):=(T x, A(x) V)$.
Clearly the $k$-domination property implies the existence
of a continuous invariant section $E^+\colon \T^d\to \Gr_k(\K^m)$  for the bundle map $\hat F$. 
The strategy to prove Theorem~\ref{teor main} is   to 
derive topological obstructions to the existence of
 continuous invariant sections $\sigma\colon \T^d\to\Gr_k(\C^m)$ of the cocycle $A$.

\begin{remark}
The statement of Theorem~\ref{teor main} hods also for $\GL_m(\R)$-valued cocycles over a torus $\T^d$ with dimension $d\geq 1$.
This can be proven   analogously or more simply using  M. Herman's  method described in~\cite{Viana-book}.
The topological obstructions there use first homotopy groups and are applicable because the real  Grassmannians 
$\Gr_k(\R^m)$ are not simply connected, something which is not true  about the complex Grassmannians $\Gr_k(\C^m)$.
\end{remark}

The paper is organized as follows. In Section~\ref{invariant sections} we provide a necessary condition for the existence of a continuous invariant section of a skew product map. In Section~\ref{consequences qp} we use the previous abstract result to provide topological obstructions to the existence of continuous invariant sections for quasi-periodic cocycles on the higher dimensional torus. This in particular implies our main theorem.

\medskip

We are grateful to Christian Sadel for posing the question regarding dominated splitting for quasi-periodic cocycles on the torus of several variables, to Gustavo Granja for a valuable suggestion on using the nonexistence of homological splitting as a topological obstruction to dominated splitting and to
Marcelo  Viana for providing us with several references on this subject.


\section{Existence of invariant sections}\label{invariant sections}

We  call {\em factor}  of linear maps   any
commutative diagram 
\begin{equation}\label{factor CD}
\begin{CD}
E    @>f >> E  \\
@V \pi  VV        @VV\pi V\\
F   @>>h > F
\end{CD} 
\end{equation}
where $E$, $F$ are vector spaces, $f\colon E\to E$, $h\colon F\to F$  are linear endomorphisms
and $\pi\colon E\to F$ is a linear epimorphism.
We call {\em splitting} of a factor~\eqref{factor CD} any linear map
$\sigma\colon F\to E$ such that $\pi\circ \sigma=\id_F$ and
$f\circ\sigma = \sigma\circ h$. In other words $\sigma$ is an $f$-invariant  section of the vector bundle $\pi\colon E\to F$.

Letting $K=\ker(\pi)$, by the fundamental theorem on homomorphisms, the linear epimorphism $\pi\colon E\to F$ induces an isomorphism $\bar \pi \colon E/K \simeq F$ through which the factor~\eqref{factor CD} can be expressed as
\begin{equation}\label{standard factor CD}
\begin{CD}
E    @>f >> E  \\
@V \pi  VV        @VV\pi V\\
E/K   @>>\bar f > E/K
\end{CD} 
\end{equation}
where $K$ stands for an  $f$-invariant vector subspace of $E$.
From these considerations it follows easily that

\begin{proposition}
\label{general splitting char}
The factor~\eqref{factor CD} has a splitting \, if and only if\,
the vector space $E$ admits an $f$-invariant decomposition
$E=G\oplus \ker(\pi)$.
\end{proposition}

Let $M$ be a compact connected manifold. Consider a continuous map $T:M\to M$ and a transitive action $G\times X\to X$ of a connected Lie group $G$ on some compact connected homogeneous space $X$.
A continuous function  $A:M\to G$ determines 
the skew-product map 
\begin{equation}
\label{F map}
 F\colon M\times X\to M\times X,\quad   F(x,p):= (T x, A(x)\,p).
\end{equation}
By definition, letting $\pi\colon M\times X\to M$ stand  for the canonical projection $\pi(x,p)=x$, the following diagram commutes

\begin{equation}\label{skew product CD}
\begin{CD}
M \times X     @>F  >> M\times X  \\
@V \pi  VV        @VV\pi  V\\
M   @>T  >> M
\end{CD} 
\end{equation}

\medskip

We call  {\em $F$-invariant section}   any  continuous
map $\sigma\colon M\to X$ such that  
$F(x,\sigma(x))=(T x,\sigma(T x))$ for all $x\in M$.

An  obvious necessary condition to the existence of an $F$-invariant section is the splitting property of the factor~\eqref{skew product CD} at the level of homology (the reader may consult~\cite{Hatcher} for a general reference on singular homology).

\begin{proposition}
\label{prop splitting}
Given a number field $\F$, if the map~\eqref{F map} admits an invariant section then for
each   $0\leq i\leq \dim M$  the homological factor
\begin{equation}\label{homology CD}
\begin{CD}
H_i(M \times X,\F)      @>F_\ast  >> H_i(M\times X,\F)   \\
@V \pi_\ast  VV        @VV\pi_\ast  V\\
H_i(M,\F)   @>T_\ast  >> H_i(M,\F) 
\end{CD} 
\end{equation}
admits a splitting.
\end{proposition}

\begin{proof}
By the K\"unneth theorem the map
$\pi_\ast\colon H_i(M\times X,\F)\to H_i(M,\F)$
is surjective.
If $\sigma\colon M\to X$ is an $F$-invariant section
then its ho\-mo\-lo\-gy $\sigma_\ast\colon H_i(M,\F)\to H_i(M\times X,\F)$
is a splitting of the homological factor~\eqref{homology CD}.
\end{proof}

The next proposition specializes the previous criterion to the case where the map $T\colon M\to M$ is homotopic to the identity.

\begin{proposition}
\label{prop no splitting}
Let $T\colon M\to M$ be a continuous map homotopic to the identity.
Let $\F$ be a number field, $A\colon M\to G$ a continuous function 
and $1\leq k \leq \dim M$ a dimension  such that:
\begin{enumerate}
\item $H_k(M,\F)\neq \{0\}$,
\item $H_i(M,\F) = \{0\}$ \, or \, $H_{k-i}(X,\F) = \{0\}$, 
for all $0<i<k$,
\item For some $p\in X$ the map
$A_p\colon M\to X$, $A_p(x):= A(x) p$, induces a non-zero homology map in dimension $k$, i.e., the linear map
$(A_p)_\ast \colon H_k(M,\F)\to H_k(X,\F)$ is non zero.
\end{enumerate}
Then $F$ admits no $F$-invariant section.
\end{proposition}

\begin{proof}
By the K\"unneth theorem and assumptions (1)-(2),
\begin{align*}
H_k(M\times X,\F) &\simeq H_k(M,\F)\otimes H_0(X,\F) \, \oplus \,
H_0(M,\F)\otimes H_k(X,\F) \\
&\simeq H_k(M,\F) \, \oplus \, H_k(X,\F) .
\end{align*}
We are also using here that $M$ and $X$ are connected so that
$H_0(M,\F)\simeq H_0(X,\F)\simeq \F$.
Hence the epimorphism  $\pi_\ast \colon H_k(M\times X,\F)\to H_k(M,\F)$ has kernel
 $$ \ker(\pi_\ast)\simeq H_0(M,\F)\otimes H_k(X,\F)
 \simeq   H_k(X,\F)  . $$
 Similarly, the projection $\pi'\colon M\times X\to X$,
 $\pi'(x,p)=p$,  induces a homology map
 $\pi'_\ast\colon H_k(M\times X,\F)\to H_k(X,\F)$ with kernel 
  $$ \ker(\pi'_\ast)\simeq H_k(M,\F)\otimes H_0(X,\F)
 \simeq   H_k(M,\F)  . $$

Because $G$ is connected, each element $A(x)\in G$ induces an action
$A(x)\colon X\to X$ which is isotopic to the identity.
Therefore the homology map $F_\ast\colon H_k(M\times X,\F) \to H_k(M\times X,\F)$
acts as the identity on $ \ker(\pi_\ast)$.

 Assume now, by contradiction, that $F$ admits an invariant section.
 By Proposition~\ref{prop splitting} there exists an $F_\ast$-invariant subspace $G$   such that 
\begin{equation}
\label{Hd splitting}
 H_k(M\times X,\F)=\ker(\pi_\ast)\oplus G.
\end{equation}
Since $T\colon M\to M$ is homotopic to $\id_M$ we have
$T_\ast=\id$ on $H_k(M,\F)$.  This implies that $F_\ast$ is the identity map on $G$. Hence, because~\eqref{Hd splitting} is $F_\ast$-invariant,  it follows that $F_\ast$ is the identity on $H_k(M\times X,\F)$.

Finally, defining the inclusion map
$i_p\colon M\to M\times X$, $i_p(x):=(x,p)$, 
since $A_p=\pi'\circ F\circ i_p$ we have at the homology level 
\begin{align*}
0\neq (A_p)_\ast &= (\pi'\circ F\circ i_p)_\ast =
 \pi'_\ast \circ F_\ast \circ (i_p)_\ast\\
 &= \pi'_\ast   \circ (i_p)_\ast = (\pi' \circ i_p)_\ast = 0 .
\end{align*} 
We have used assumption (3)  and the fact that the composition $\pi' \circ i_p$ is a constant map. This contradiction proves that $F$ admits no invariant section.
\end{proof}


\section{Consequences for quasi-periodic cocycles}\label{consequences qp}

Finally we show that for certain homotopy types
a continuous quasi-periodic cocycle $A\colon\T^d\to \GL_m(\C)$ cannot have
dominated splitting. The base dynamics is a\-ssumed to be an ergodic  translation of a torus $\T^d$ of dimension $d\geq 2$.

Let $\Gr_k(\C^m) $ denote the  {\em complex Grassmannian} of
$k$-dimensional complex subspaces of $\C^m$.

\begin{proposition}
\label{coro Td GLm}
Let  $A\colon\T^d\to \GL_m(\C)$  be a continuous function with $d\geq 2$.
If  the map
$A_{V}:\T^d \to \Gr_k(\C^m)$,
$A_{V}(x)=A(x) V$,  for some $V\in \Gr_k(\C^m)$, induces a non-zero homology map in dimension two, i.e., the linear map
$(A_{V})_\ast \colon H_2(\T^d,\F)\to H_2(\Gr_k(\C^m),\F)$ is non zero for some field $\F$ and some $1\leq k <m$,
then the quasi-periodic cocycle $A$ has no continuous invariant section $\sigma\colon \T^d\to \Gr_k(\C^m)$. In particular  $A$  is not $k$-dominated.
\end{proposition}

\begin{proof}
Let us apply Proposition~\ref{prop no splitting} with
$M=\T^d$, $X=\Gr_k(\C^m)$ and dimension $k=2$. 
For any field $\F$  we have
$\dim H_0(\Gr_k(\C^m),\F)=1$ because $\Gr_k(\C^m)$ is a connected manifold. We have  $\dim H_1(\Gr_k(\C^m),\F)=0$ and
$ \dim H_2(\Gr_k(\C^m),\F) \geq 1$  
(see~\cite[Section 3.2]{Nicolaescu} or~\cite[Section 5 of Chapter 1]{Harris-Griffiths}).
We also have  $H_2(\T^d,\F)\simeq \F^{\binom{d}{2}}\neq \{0\}$
because $d\geq 2$. Therefore assumption  (1) and (2) of Proposition~\ref{prop no splitting} hold.
On the other hand, our hypothesis implies assumption (3)
of that proposition. Therefore, by Proposition~\ref{prop no splitting}, the map $\hat F\colon \T^d\times \Gr_k(\C^m)\to \T^d\times \Gr_k(\C^m)$
does not admit any $\hat F$-invariant section.

Finally, if the quasi-periodic cocycle $A$ is $k$-dominated 
then the $F$-invariant sub-bundle $E^+$   determines an $\hat F$-invariant section  $E^+:\T^d\to \Gr_k(\C^m)$. This contradiction proves that
$A$ is not $k$-dominated.
\end{proof}

\begin{corollary}
\label{coro T2 GL2}
Consider a quasi-periodic cocycle
$A\colon \T^2\to\GL_2(\C)$. 
If  the map
$A_{\hat v}:\T^2 \to \Pp(\C^2)$,
$A_{\hat v}(x)=A(x) \hat v$, for some $\hat v\in \Pp(\C^2)$, is not homotopic to a constant 
then   $A$ does not have dominated splitting.
\end{corollary}

\begin{proof}
The projective space $\Pp(\C^2)$ can be identified with the Riemann sphere $\Su^2\equiv \C\cup \{\infty\}$.
Since $A_{\hat v}$ is not homotopic to a constant, by Hopf theorem $\deg( A_{\hat v} )\neq 0$. Then, making the canonical identifications $H_2(\T^2,\F)\simeq \F$
and $H_2(\Pp(\C^2),\F)\simeq \F$, the homology map
$(A_{\hat v})_\ast\colon H_2(\T^2,\F)\to H_2(\Pp(\C^2),\F)$ 
is the multiplication by $\deg( A_{\hat v} )$,
and hence it is non zero.
\end{proof}

\medskip

\begin{corollary}
\label{coro T2 GL2 no DS}
There are  continuous functions 
$A\colon \T^2\to\GL_2(\C)$ whose homotopy classes contain no quasi-periodic cocycle with do\-mi\-na\-ted splitting.
\end{corollary}

\begin{proof}
Consider any analytic map $f\colon\T^2\to\R^3\setminus\{0\}$. Let $p\colon \R^3\setminus\{0\}\to\Pp(\C^2)$ be the composition of the projection 
$$\R^3\setminus\{0\} \ni (x,y,z)\mapsto \frac{(x,y,z)}{\norm{(x,y,z)}}\in\Su^2$$
onto the unit sphere $\Su^2$ with the stereographic projection, which maps $\Su^2$  diffeomorphically onto
 the projective space $\Pp(\C^2)=\C\cup \{\infty\}$.

\medskip

Assume that the  origin belongs to a bounded connected component of $\R^3\setminus f(\T^2)$.
Then the parametric hypersurface $f$ has non zero winding number around $0$, which implies that the composition $\phi=p\circ f\colon \T^2 \to \Pp(\C^2)$ has non zero degree. 

\medskip

Write $\phi=a/b$ as the ratio  of two real
analytic functions $a,b\colon \T^2\to \C$, where $b$ vanishes exactly at the points $x\in\T^2$ where $\phi(x)=\infty$ and the pair $(a(x),b(x))\neq (0,0)$
for all $x\in\T^2$. Then the analytic function
$A\colon\T^2\to\GL_2(\C)$ 
$$ A(x)=\begin{pmatrix}
a(x) & -\overline{b(x)} \\
b(x) & \overline{a(x)}
\end{pmatrix}$$
satisfies the assumption of Corollary~\ref{coro T2 GL2}
with $\hat v=\widehat{(1,0)}$.
Hence it cannot have dominated splitting.

\medskip

Finally, if $B\colon \T^2\to\GL_2(\C)$ is another 
continuous function homotopic to $A$ then the functions
$A_{\hat v} \colon \T^2 \to \Pp(\C^2)$, 
$A_{\hat v}(x)=A(x)\hat v$,
and $B_{\hat v}\colon \T^2 \to \Pp(\C^2)$,
$B_{\hat v}(x)=B(x)\hat v$, are also homotopic.
Hence $B_{\hat v}$ is not homotopic to a constant
and by  Corollary~\ref{coro T2 GL2} the cocycle
$B$ cannot have dominated splitting either.
\end{proof}

Theorem~\ref{teor main} follows from the following proposition.

\medskip

\begin{proposition}
\label{prop No DS d>1}
Given integers $d\geq 2$ and $1\leq k <m$
there exist analytic quasi-periodic cocycles
$A\colon \T^d\to \GL_m(\C)$ with an invariant 
measurable decomposition $\C^m=E^+\oplus E^-$
such that
\begin{enumerate}
\item $\dim E^+=k$,
\item all Lyapunov exponents of $A\vert_{E^+}$ are positive,
\item all Lyapunov exponents of $A\vert_{E^-}$ are negative,
\item no continuous cocycle $B\colon \T^d\to \GL_m(\C)$ in the homotopy class of  $A$ admits a continuous invariant section $\sigma\colon \T^d\to\Gr_k(\C^m)$.
\end{enumerate}   
\end{proposition}

\begin{proof}
Consider an analytic quasi-periodic cocycle $A\colon \T^2\to \SL_2(\C)$ in the homotopy class of a cocycle 
given by Corollary~\ref{coro T2 GL2 no DS}.
By~\cite[Theorem 1]{A11}, possibly perturbing it,
 we can assume that $A$ admits an invariant 
measurable decomposition $\C^m=E^+\oplus E^-$
with $\dim E^- = \dim E^+=1$ and having non-zero Lyapunov exponents,
w.r.t. an ergodic translation with frequency vector $\omega\in\T^2$. 

Take positive numbers $\lambda>1>\mu$ such that
$\lambda^{k-1}\,\mu^{m-k-1}=1$ and let $\tilde{A}\colon \T^d\to \SL_m(\C)$ be the cocycle
$$ \tilde{A}(x):=
\begin{pmatrix}
\lambda\, I_{k-1} & 0 &  0\\
0 & A(\pi(x)) & 0 \\
0 & 0 & \mu\,I_{m-k-1} 
\end{pmatrix} \in \SL_m(\C) $$
where $I_{k-1}$ and $I_{m-k-1} $ stand for identity matrices of the specified dimensions, and $\pi\colon \T^d\to \T^2$ denotes the projection   $\pi(x_1,\ldots, x_d):=(x_1,x_2)$. By construction the cocycle $\tilde{A}\colon \T^d\to \SL_m(\C)$ satisfies properties (1)-(3), w.r.t.  any ergodic translation with frequency vector $\tilde{\omega}\in\T^d$ such that $\pi(\tilde{\omega})=\omega$.

We are going to use Proposition~\ref{prop no splitting} to prove item (4). For each $1\leq i\leq m$, let $V_i\in\Gr_k(\C^m)$ be the complex $i$-plane generated by the first $i$ vectors $e_1,\ldots, e_i$ of the  canonical basis of $\C^m$.
We claim that the map $\tilde A_{V_k}\colon \T^d\to \Gr_k(\C^m)$,
$\tilde A_{V_k}(x):=\tilde{A}(x)\, V_k$, induces a non zero linear map
$(\tilde A_{V_k})_\ast\colon H_2(\T^d,\K) \to H_2(\Gr_k(\C^m),\K)$. By construction this is true about the map $A_{e_1}\colon \T^2\to \Gr_k(\C^2)$,
$A_{e_1}(x):=A(x)\, \hat e_1$, which induces a non zero linear map
 at the second homology level.
To relate the homologies of $\tilde{A}_{V_k}$ and $A_{e_1}$
we  factor the first, $\tilde{A}_{V_k}$, as a composition of several maps which include the second, $A_{e_1}$.

Let $\Sigma:=\{V\in\Gr_k(\C^m)\,\colon \,
V_{k-1}\subset V\subset V_{k+1}\,\}$. This is a complex analytic submanifold
of the Grassmannian space $\Gr_k(\C^m)$, which is diffeomorphic to the complex projective line  $\Pp(\C^2)$.
Let $p\colon \C^m\to \C^2$ be the linear projection
$p(z_1,\ldots, z_m)=(z_{k},z_{k+1})$, and define 
$H\colon \Pp(\C^2)\to \Sigma$ by $H(\hat v):= V_{k-1} + \C\,p(v)$.
Then for all $x\in\T^d$,
\begin{align*}
\tilde{A}_{V_k}(x)&= \tilde{A}(x)\, V_k =
V_{k-1} + \C\,p(A(\pi(x))\, e_1) \\
&=H( A_{e_1}(\pi(x)))=(\iota\circ H\circ A_{e_1}\circ \pi)(x)
\end{align*}
where $\iota$ stands for the inclusion map $\iota\colon \Sigma\to \Gr_k(\C^m)$.

By K\"unneth theorem, the linear map
$\pi_\ast\colon H_2(\T^d,\F)\to H_2(\T^2,\F)$ is surjective. 
Because $H$ is a diffeomorphism, the homology map $H_\ast$ is an isomorphism. We are left to prove that 
$\iota_\ast \colon  H_2(\Sigma,\K)\to H_2(\Gr_k(\C^m),\K)$
is  injective. This will imply that
$(\tilde{A}_{V_k})_\ast$ is non zero and, by Proposition~\ref{prop no splitting}, that  no cocycle homotopic to 
$\tilde A$ admits a continuous invariant section with values in $\Gr_k(\C^m)$. 

Let us now turn to prove the injectivity of
$\iota_\ast$. The Grassmannian $\Gr_k(\C^m)$ is an analytic manifold of dimension $k\, (m-k)$. By Schubert Calculus (see~\cite[Section 3.2]{Nicolaescu} or~\cite[Section 5 of Chapter 1]{Harris-Griffiths}), the manifold
$\Gr_k(\C^m)$ admits a class of standard cell decompositions, whose cells are referred as {\em Schubert cells}. The closures of these cells  are analytic subvarieties known  as {\em Schubert cycles}.
The submanifold $\Sigma$ is itself  a Schubert cycle with complex dimension $1$ which can be integrated in a cell decomposition
$$ \{V_k\}=\Sigma_0 \subset \Sigma =\Sigma_1 \subset \Sigma_2 \subset \cdots \subset \Sigma_{N}=\Gr_k(\C^m) . $$
Each space $\Sigma_{i}$ is an analytic subvariety obtained from $\Sigma_{i-1}$ by joining a cell with (real) even dimension and boundary contained in $\Sigma_{i-1}$. This implies that
$H_1(\Sigma_i,\Sigma_{i-1},\K)=0$ for all $i$ and all fields $\K$. Hence, by the long exact sequence of the pair $(\Sigma_i,\Sigma_{i-1})$,
$$ 0= H_1(\Sigma_i,\Sigma_{i-1},\K)\to H_2(\Sigma_{i-1},\K)
\to H_2(\Sigma_{i},\K) \to \cdots $$
is an exact sequence. Therefore, because $\iota$ can be factored
as the composition of the inclusions
$\Sigma_{i-1} \hookrightarrow \Sigma_i$ with $i=2,\ldots, N$,
the map $\iota$ is injective at the second homology level.
\end{proof}

\begin{remark}\label{related to continuity LE}
Given a cocycle $A\in C^{\om}_{r} (\T^d, \GL_m(\R))$ in one of  the homotopy classes from Proposition~\ref{prop No DS d>1},
the cocycle $A_y:\T^d\to \GL_m(\C)$, $A_y(x)=A(x+i y)$,
 {\em cannot} have dominated splitting for any $y\in\R^d$.
\end{remark}

\subsection*{Acknowledgments}

The first author was supported  by Funda\c{c}\~{a}o para a Ci\^{e}ncia e a Tecnologia, under the project: UID/MAT/04561/2013.

The second author was supported by the Norwegian Research Council project no. 213638, ``Discrete Models in Mathematical Analysis'' and by the Conselho Nacional de Desenvolvimento Cient\'{i}fico e Tecnol\'{o}gico (CNPq, Brazil).

\def\cprime{$'$} \def\cprime{$'$}
\providecommand{\bysame}{\leavevmode\hbox to3em{\hrulefill}\thinspace}
\providecommand{\MR}{\relax\ifhmode\unskip\space\fi MR }
\providecommand{\MRhref}[2]{%
  \href{http://www.ams.org/mathscinet-getitem?mr=#1}{#2}
}
\providecommand{\href}[2]{#2}

\end{document}